\documentclass[letterpaper,12pt]{article}
\usepackage{amsmath,amsthm,amsfonts}
\usepackage{amssymb}
\usepackage{setspace}
\usepackage{lineno} 

\usepackage[latin1]{inputenc}

\newtheorem{theorem}{Theorem}

\newtheorem{proposition}{Proposition}
\newtheorem{corollary}{Corollary}
\theoremstyle{remark}
\newtheorem{remark}{Remark}

\onehalfspace
\def\nn{{\mathbb N}}

\def\EE{{\mathcal E}}  
 \def\II{{\mathcal I}}

\title{{\bf On the number of $m$th roots of permutations}
\footnotetext{Rivera-Mart\'{i}nez was  partially supported by PROMEP (SEP), grant UAZ-PTC-103.}
}
\author{
Jes\'{u}s Lea\~{n}os, 
Rutilo Moreno and Luis M. Rivera-Mart\'{i}nez\\
}
\date{}

\begin{document}

\maketitle
\begin{abstract}
  Let $m$ be a fixed positive integer. It is well-known that a permutation $\sigma$ of $\{1,...,n\}$ may have one, many, or no $m$th roots. In this article we provide an explicit expression and a generating function for the number of $m$th roots of $\sigma$. Let $p_m(n)$ be the probability that a random $n$-permutation has an $m$th root. We also include a proof of the fact that $p_m(jq)=p_m(jq+1)=\cdots =p_m(jq+(q-1))$, $j=0,1,...$, when $m$ is a power of prime number $q$.
  \end{abstract}
\section{Introduction and main results}
Let $S_n$ be the group of all permutations of the finite set $[n]=\{1,...,n\}$. Let $m$ be a fixed positive integer. We say that $\sigma \in S_n$ has an $m$th root or that $\sigma$ is an $m$th power if there exists a permutation $\tau \ \in S_n$ with $\tau^m=\sigma$. For fixed $m$, not all permutations have an $m$th root (\cite{wilf}, Theorem 4.8.2), however, L. Glebsky and L. M. Rivera \cite{glebskyrivera} have proved that for sufficiently large $n$, any permutation has an ``almost" $m$th root (in the sense of the Hamming distance \cite{dezahuang}). Now, if we know that a permutation $\sigma$ has at least one $m$th root, how many $m$th roots can $\sigma$ have? We can find an explicit expression for this quantity in the paper of A. I. Pavlov \cite{pavlov1}.  Also, in the article of S. Annin, T. Jansen and C. Smith~\cite{annin}, appeared a classification of the elements in $S_n$ and $A_n$ that has $m$th roots, and they propose some problems related with the roots of permutations. In particular, our work is about some questions on Problem 1 in Section 4 of \cite{annin}.

The main results of this paper are an explicit expression for the number of $m$th roots of any $n$-permutation $\sigma$ (Theorem~\ref{the1}) and a generating function for this number (Theorem~\ref{the2}). In order to obtain our expression we define some sets of non-negative integers that seem interesting by themselves (see Section~\ref{sec2}). In particular, such sets provide a simpler expression than the corresponding expression in \cite{pavlov1}. Moreover, this new expression allows us to compute the number of $m$th roots of a permutation in an effective way using a computer algebra system.

Another classical problem consists in estimating the number of permutations in $S_n$ that admit an $m$th root. P. Tur\'{a}n \cite{turan} gave an upper bound when $m$ is a prime number and Blum~\cite{blum} gave an asymptotic formula for the case $m=2$. Recently, M. B\'{o}na, A. McLennan and D. White \cite{bona} proved that the probability that a random permutation of length $n$ has an $m$th root with $m$ prime, is monotonically non-increasing in $n$. See also the work of N. Pouyanne~\cite{pouyanne}  for an asymptotic study for any positive integer $m$ and the work of B. Bollob\'{a}s and B. Pittel \cite{bollobas} who continued the work of N. Pouyanne and studied the limiting distribution of the root degree of a permutation. This problem can be easily reformulated as the problem about the probability, $p_m(n)$, that a $n$-permutation chosen uniformly at random has an $m$th root. For this problem, we give a proof of the fact that when $m$ is a power of a prime $q$, for all $j \geq 0$, $p_m(jq)=p_m(jq+1)=\cdots=p_m(jq+(q-1))$. For the case $m$ a prime, see the paper of M. B\'{o}na, et al. \cite{bona} that includes a combinatorial proof of the equivalent equalities. It is also recommended the paper of A. Mar\'{o}ti \cite{attila} and the bibliography therein for related results about the proportion of $\ell$-regular elements in the symmetric group $S_n$. Another interesting article is due to M. R. Pournaki~\cite{pourn}, who worked in the problem of determining the number of even permutations with roots.

Before stating our main results, we shall give some notation and definitions. As usual, we denote by $\nn$ (respectively $\nn_0$) the set of positive (respectively, non-negative) integers. The {\it cycle type} of an $n$-permutation is a vector ${\bf a}=(a_1,a_2,...,a_n)$ that indicates that the permutation has $a_i$ cycles of length $i$ for every $i \in [n]$. It is an easy exercise to prove that conjugated permutations have the same number of $m$th roots. Let $m, \ell \in \nn$ and $a \in \nn_0$. We define the following sets
\begin{eqnarray*}
 G_m(\ell,a)  := \{ g \ \in \nn : \; g \leq a; \; \gcd\left(g\ell, m\right) = g  \},
\end{eqnarray*}
and
\begin{eqnarray*}
 G_m(\ell):= \{ g \ \in \nn : \; \; \gcd\left(g\ell, m\right) = g  \}.
\end{eqnarray*}
Clearly $G_m(\ell,a) \subseteq G_m(\ell)$, and both are finite sets. Note that if $a=0$, then $G_m(\ell, a)=\emptyset$.  We will name $(g_1,...,g_k)$ the {\it associate vector} of $G_m (\ell, a)$ if  $G_m (\ell, a)=\{g_1, \cdots, g_k\}$ and $g_1 < g_2 < \cdots< g_k$. For every set $G_m(\ell,a)$ of cardinality $k \geq 1$ we define the set of vectors
\begin{eqnarray*}
 \EE_m(\ell,a) := \{ {\bf \varepsilon} \ \in {\nn}_0^k \; : \; {\bf g} \cdot {\bf \varepsilon} = a,\;  \; \text{with {\bf g} the associate vector of}\; G_m(\ell,a) \}.
\end{eqnarray*}
Note that if equation $g_1x_1+\cdots +g_kx_k=a$ does not have non-negative integer solutions then $\EE_m(\ell,a)=\emptyset$. We use the convention that $\displaystyle \sum_{i \in \II} s_i=0$ if $\II$ is empty.

Next, we present our main results about the number of $m$th roots of permutations.
\begin{theorem}\label{the1}
 Let $m$ be a fixed positive integer. Let $\sigma$ be any $n$-permutation of type ${\bf a}$, i.e. for every $\ell \in [n]$, $\sigma$ has $a_\ell$ cycles of length $\ell$. Let $r^{(m)}({\bf a})$ be the number of $m$th roots of $\sigma$, then
\begin{eqnarray*}
 r^{(m)}({\bf a})=\prod_{\substack{\ell \geq 1\\a_\ell \neq 0}} a_\ell ! \left( \sum_{ {\bf \varepsilon} \ \in \EE_m(\ell,a_\ell)} \    \prod_{i=1}^k \frac{\ell^{(g_i - 1)\varepsilon_i}}{ {g_i}^{\varepsilon_i}\varepsilon_i!}   \right),
\end{eqnarray*}
where $k=|G_m(\ell,a_\ell)|$, and ${\bf g}=(g_1,...,g_k)$ is the associate vector of $G_m(\ell,a_\ell)$.
\end{theorem}

Our next theorem provides a generating function for the number of $m$th roots of permutations.
\begin{theorem}\label{the2} Let $m,n$ be positive integers, and $a_1,...,a_n$ be non-negative integers. For $n=a_1+2a_2+\cdots +na_n$, the coefficient of $\frac{t_1^{a_1} \cdots t_{n}^{a_n}}{a_1! \cdots a_n!}$ in the expansion of
\begin{equation*}
\exp\Big(\sum_{\ell\geq1}\sum_{g \in G_m(\ell)}\frac{\ell^{g-1}}{g}t_\ell^g\Big),
\end{equation*}
is the number of $m$th roots of an $n$-permutation of cycle type ${\bf a}=(a_1,...,a_n)$.
\end{theorem}
 Note that Theorems~\ref{the1} and~\ref{the2} allow us to compute the number of $m$th roots of a permutation in an effective way using a computer algebra system. 
 
The outline of the paper is as follows. In Section~\ref{sec2} we present some preliminaries results about the sets $G_m(\ell, a)$ and $G_m(\ell)$. In Section~\ref{sec3} we recall how to extract roots of permutations. In Section~\ref{sec4} we give the proofs of Theorem~\ref{the1} and Theorem~\ref{the2}. Finally, in Section~\ref{sec6} we present a proof of the fact that probability $p_m(n)$ satisfies $p_m(jq)=p_m(jq+1)=\cdots =p_m(jq+(q-1))$, where $j=1,...$ and $m$ is a power of a prime number $q$.

\section{Preliminaries}\label{sec2}
The next notation is standard, and it can be found in several books, for example, in the book of V. Schoup (\cite{schoup}, Section 1.3). For integer $n$ and prime $p$ we will write $\nu_p(n)$ for the highest power of $p$ that divides $n$. In this notation the definition of $\gcd(a,b)$ is
$$
\gcd(a,b)=\prod_{p \; \in \mathcal{P}} p^{\text{min}(\nu_p(a),\nu_p(b))},
$$
where $\mathcal{P}$ is the set of all primes. We use the convention that $\gcd(g)=g$ for every $g\in \nn$. If $a$ and $b$ are positive integers then
$$
\nu_p(\gcd(a,b))=\text{min}(\nu_p(a),\nu_p(b)),
$$
and
$$
\nu_p(a\cdot b)=\nu_p(a)+\nu_p(b).
$$
Note that $a$ divides $b$ if and only if $\nu_p(a) \leq \nu_p(b)$ for all primes $p$. The following definition can be found in the book of H. Wilf (\cite{wilf}, page 148). For a pair $\ell$, $m$ in $\nn$, the number $((\ell,m))$ is defined as
$$
((\ell, m))=\prod_{\substack{p ~|~\ell \\ p \; \in \mathcal{P} }}p^{\nu_p(m)}.
$$
The number $((\ell,m))$ is very important in the characterization of the permutations that admit $m$th roots (Theorem~\ref{root}, Section~\ref{sec3}). As the sets $G_m(\ell, a)$,  and $G_m(\ell)$ play an important role in our expressions, we first prove some interesting propositions about the elements in these sets. Some of these propositions show the relation between the  elements of  $G_m(\ell, a)$ and $G_m(\ell)$ with the number $((\ell,m))$.

\begin{proposition}\label{gequal1}
If $a \geq1$ then $1 \in G_m(\ell, a)$ if and only if $\gcd(\ell, m)=1$.
 \end{proposition}
 \begin{proof}
   Because $1=\gcd(\ell,m)=\gcd(1 \cdot \ell, m)$, we have $1 \in  G_m(\ell, a)$. The converse follows from the definition of $G_m(\ell, a)$.
 \end{proof}
 
 \begin{proposition}\label{typeg2} Let $g$, $\ell$ and $m$ be positive integers. Then $g \in G_m(\ell)$ if and only if conditions 1 and 2 hold
 \begin{enumerate}
 \item Any prime divisor $p$ of $g$ divides $m$ and satisfies one of the following:\label{c1}
\begin{enumerate}
\item If $p$ divides $\ell$, then $\nu_p(g)=\nu_p(m)$, \label{p2}
\item If $p$ does not divide $\ell$, then $\nu_p(g)\leq \nu_p(m)$. \label{p1}
\end{enumerate}
\item If $p$ is a prime that does not divide $g$, then $p$ does not divide $\gcd(\ell, m)$.\label{p3}
\end{enumerate}
\end{proposition}
\begin{proof}
For the if part: As any prime divisor $p$ of $g$ satisfies (\ref{c1})  $\nu_p(g) \leq \nu_p(m)$, and therefore $g~|~m$ and $g~|~\gcd(g\ell, m)$. Now we prove that $\gcd(g \ell, m)~|~g$. Let $p$ be any prime divisor of $\gcd(g \ell, m)$, then $p~|~g\ell$ and $p~|~m$. If $p \nmid g$ then $p~|~\ell$ and $p~|~\gcd(\ell, m)$, which implies a contradiction of condition (\ref{p3}), so $p~|~g$.
If $p$ divides $\ell$ then by (\ref{p2}), $\nu_p(g)=\nu_p(m)$. Thus $$\nu_p(\gcd(g\ell,m))=\text{min}(\nu_p(g\ell), \nu_p(g))=\nu_p(g).$$
If $p$ does not divide $\ell$, $\nu_p(\ell)=0$, and by property (\ref{p1}) $\nu_p(g)\leq \nu_p(m)$. Then $$\nu_p(\gcd(g\ell,m))=\text{min}(\nu_p(g)+\nu_p(\ell), \nu_p(m))=\nu_p(g).$$
Therefore $\gcd(g\ell,m)~|~g$, and $g \in G_m(\ell)$. \\
For the converse. Let $p$ be any prime divisor of $g \in G_m(\ell)$. By definition $g=\gcd(g\ell,m)$ and then $p~|~m$. We first prove (\ref{p2}). The hypothesis  $p~|~\ell$ implies $\nu_p(\ell) > 0$, and  $\nu_p(g)=\text{min}(\nu_p(g)+\nu_p(\ell),\nu_p(m))=\nu_p(m)$. Now we prove (\ref{p1}). As $p \nmid \ell$, $\nu_p(\ell)=0$ and $\nu_p(g)=\text{min}(\nu_p(g)+\nu_p(\ell),\nu_p(m))$, and therefore $\nu_p(g) \leq \nu_p(m)$. Finally, we prove (\ref{p3}), Suppose that there exists a prime $p$ such that $p \nmid~g$ and $p~|~\gcd(\ell,m)$. Then $p~|~\gcd(g \ell, m)=g$, which is clearly a contradiction.
\end{proof}
\begin{remark}
There is another way to build the set $G_m(\ell)$: choose any divisor $d \geq 1$ of $m$, with $d$ relatively prime to $\ell$, and define $g:=m/d$. Thus, we have that $\gcd(\ell g, m)=\gcd(\ell g,gd)=g$, as it is required. Then, the set  $G_m(\ell)$ can be defined as the set $\{g=m/d\;  : \; d \in \nn, d~|~m \text{ and} \gcd(d, \ell)=1 \}$. It is an easy exercise to show that both sets are the same. To build the set $G_m(\ell,a)$ follow the same procedure only with the condition that  $g=m/d \leq a$.

\end{remark}\label{re0}
The following proposition shows that $((\ell,m))$ is an element of $G_m(\ell)$  
\begin{proposition}\label{pr2}
Let $\ell$, $m$, $a$ be positive integers. Then
\begin{enumerate}
\item $((\ell, m))$  belongs to  $G_m(\ell)$,\label{pr21}
\item if $((\ell, m))$ divides $a$, then $((\ell,m)) \in G_m(\ell, a)$.\label{pr22}
\end{enumerate}
\end{proposition}
\begin{proof}
 First we show~\ref{pr21}. The case $\gcd(\ell,m)=1$ follows from Proposition~\ref{gequal1}. If $\gcd(\ell,m) > 1$, for any prime divisor $p$ of $((\ell, m))$, we have that $p$ divides $\gcd(\ell, m)$ and $p$ divides $\ell$ and $m$. By definition of $((\ell, m))$, $\nu_p\big(((\ell, m))\big)=\nu_p(m)$, and condition (1) in Proposition~\ref{typeg2} holds. Let $q$ be a prime that does not divide $((\ell, m))$, if $q \nmid \ell$ clearly $q \nmid \gcd(\ell, m)$ and if $q \mid \ell$, use the hypothesis that $ q \nmid ((\ell, m))$ and the definition of $((\ell, m))$ to obtain $\nu_q(m)=0$. Therefore, condition (2) in Proposition~\ref{typeg2} also holds, and $((\ell, m)) \in G_m(\ell)$. Now we show the second part. As a consequence of the first part of this proposition, $((\ell, m)) \in G_m(\ell)$. As $((\ell, m))$ divides $a$, $((\ell, m)) \leq a$, and therefore $((\ell, m)) \in G_m(\ell, a)$.\\
\end{proof}

\begin{remark}\label{re1}
Note that if $k=\gcd(k\ell,m)$ then $\gcd(\ell,m)$ divides $k$.
\end{remark}

\begin{proposition}\label{propc}
Let $\ell, m$ be positive integers. If  $G_m(\ell)=\{g_1,...,g_h\}$, then $\gcd(g_1,...,g_h)=((\ell,m))$. \label{part0}

\end{proposition}
\begin{proof}
We begin with the case $\gcd(g_1,...,g_h)=1$. As a consequence of  Remark~\ref{re1}, if $g_i=\gcd(g_i\ell,m)$ then $\gcd(\ell,m)$ divides $g_i$ for any $i \in [h]$, and therefore we have $\gcd(\ell,m)=1$ which implies that $((\ell, m))=1$. Now, for $d:=\gcd(g_1,...,g_h) > 1$, from part (\ref{pr21}) of Proposition~\ref{pr2}, $((\ell, m)) \in G_m(\ell)$, and therefore $d$ divides $((\ell, m))$. Now we prove that  $((\ell, m))$ divides $d$. As $d > 1$ and $d~|~((\ell, m))$ then $((\ell, m)) > 1$. Let $p$ be any prime factor of $((\ell, m))$. From the  definition of $((\ell,m))$, $p$ divides $\ell$ and $m$, thus $p$ divides $\gcd(\ell, m)$. Let $g$ be any element in $G_m(\ell)$. So $g=\gcd(g\ell,m)$. By Remark 1, we know that $\gcd(\ell, m)$ divides $\gcd(g\ell, m)$ and therefore $p$ divides $g$. By Proposition~\ref{typeg2} (\ref{p2}), the exponent of $p$ in any $g \in G_m(\ell)$ is $\nu_p(m)$, then $\nu_p(d)=\nu_p(m)$. On the other hand, by definition $\nu_p\big(((\ell, m))\big)=\nu_p(m)$, and therefore $\nu_p\big(((\ell, m))\big)=\nu_p(d)$, so we conclude that $((\ell, m))$ divides $d$.
\end{proof}
A consequence of this proposition is that $((\ell, m))$ is the least element of the set $G_m(\ell)$. We also obtain the following corollary
\begin{corollary}\label{coro1}
Let $a, \ell, m$ be positive integers. Let $G_m(\ell , a)=\{g_1,...,g_j\}$. Then 
$$\gcd(g_1,...,g_j)=((\ell,m)).$$
\end{corollary}

 \begin{proposition}\label{props}
    Let $m, \ell, a$ be positive integers. Let $G_m(\ell, a)=\{g_1,...,g_h\}$. Then $((\ell,m))$ divides $a$ if and only if
    \begin{equation}\label{leq}
     g_1x_1+g_2x_2+ \cdots +g_hx_h=a,
     \end{equation} has non-negative integer solutions.
     \end{proposition}
    \begin{proof} If $h=1$, then $G_m(\ell, a)=\{((\ell, m))\}$ by  Proposition~\ref{propc},  and the result follows. Now, we assume $h > 1$. The if part follows easily because $((\ell, m))=\gcd(g_1, ..., g_h)$ (Corollary~\ref{coro1}). For the converse, by Proposition~\ref{pr2} part~(\ref{pr22}), $((\ell, m)) \in \{g_1,..., g_h\}$. Now, take $g_i=((\ell, m))$, $x_i=a/((\ell, m))$ and $x_j=0$ for any $j \neq i$. \end{proof}

We conclude this section with one proposition about the elements of $G_m(\ell)$ for the case when $\ell$ and $m$ are relatively prime.
\begin{proposition}\label{typeg1}
If $\gcd(\ell, m)=1$, then $G_m(\ell)$ is equal to the set of positive divisors of $m$.
\end{proposition}
\begin{proof}
As $\gcd(\ell,m)=1$ then $\gcd(g\ell,m)=\gcd(g,m)$. Any $g$ that satisfies the relation  $g=\gcd(g\ell,m)=\gcd(g,m)$ is a positive divisor of $m$.
\end{proof}

\section{Roots of permutations}\label{sec3}
Let $\sigma$ be an $n$-permutation and let $m$ be a fixed positive integer. A permutation $\tau \ \in \ S_n$ that satisfies $\tau^m=\sigma$ may or may not exist.  If such $\tau$ exists, it is called an $m$th root of $\sigma$. The following theorem is due to A. Knopfmacher and R. Warlimont  \cite[p. 148, Theorem 4.8.2]{wilf}, and it gives a characterization of the $n$-permutations that have  $m$th roots. 
\begin{theorem}\label{root}  Let $m$ be a positive integer. A permutation $\sigma$ has an $m$th root if and only if for every $\ell=1,2,...$ the number of $\ell$-cycles that $\sigma$ has is divisible by $((\ell,m))$.
\end{theorem}
This theorem gives us information about the cycle structure of permutations that admits $m$th roots. But, if we know that a permutation has $m$th roots, we need a procedure  to get all of its roots (see, for example  \cite[\S 3.3]{ghs}). The first important observation is that we can build any $m$th root $\tau$ of $\sigma$ if we work with cycles of $\sigma$ with different lengths separately.  Indeed, if $\tau$ is an $m$th root of $\sigma$, with $\tau=C_1 \cdots C_s$ its disjoint cycle factorization, then $\sigma=\tau^m=C_1^m \cdots C_s^m$, where $C_i^m$ consists in $\gcd(|C_i|, m)$ cycles of length $|C_i|/\gcd(|C_i|, m)$ each.  Let $\sigma_\ell$ be the part of $\sigma$, that consists in the product of  all cycles $D_i$ of length $\ell$ in $\sigma$ ($a_\ell\neq 0$, where $a_\ell$ is the number of cycles of length $\ell$ in $\sigma$). So, we can see permutation $\sigma$ as a product of parts $\sigma_\ell$, where every part $\sigma_\ell$ consist in all the cycles of length $\ell$ in the disjoint cycle factorization of $\sigma$.
Let $\tau_\ell$ be the part of $\tau$ whose $m$th power produces the part $\sigma_\ell$
$$
\tau_\ell^m=\sigma_\ell=\prod_{\mid D_i \mid=\ell} D_i.
$$
and we can find $\tau_\ell$, and therefore $\tau$, working with $\sigma_\ell$ independently. This is for every length $\ell \neq 0$ in cycles of $\sigma$. Now, we need to find the admissible lengths, $|C_i|$, of cycles in $\tau_\ell$. We know that for any cycle $C_i$ in $\tau_\ell$, $|C_i|=\ell \gcd(|C_i|, m)$, and we write $g=\gcd(|C_i|, m)$. We need to find $g \leq a_\ell$, in such a way that $g=\gcd(g\ell,m)$. Note that this number exists because we are assuming that the permutation has $m$th roots, and by Theorem~\ref{root} and Proposition~\ref{pr2} (part~\ref{pr22}), $((\ell, m)) \in G_m(\ell, a_\ell)$. Thus, any $g$ cycles of $\sigma_\ell$, $D_{i_1}, \cdots, D_{i_g}$, can be combined in a suitable way (see, for example, proof of Theorem 3 in \cite{annin} or proof of Theorem 4.8.2 in \cite{wilf}) into one $g\ell$-cycle $C_i$ for $\tau_\ell$ with $C_i^m=D_{i_1} \cdots D_{i_g}$. 

Now, to obtain all the $m$th roots of $\sigma$, we build the sets $G_m(\ell, a_\ell)$ and $\EE_m(\ell,a_\ell)$, for every $\ell$ with $a_\ell \neq0$. Theorem~\ref{root} and Proposition~\ref{props} shows that $\EE_m(\ell, a_\ell)$ is not empty if and only if $((\ell, m))$ divides $a_\ell$. The elements of $\EE_m(\ell,a_\ell)$  represents the different ways in which we can group the cycles of $\sigma_\ell$. This is, for $\varepsilon \in \EE_m(\ell,a_\ell)$, the coordinates $\varepsilon_1,\varepsilon_2, \dots$ of $\varepsilon$ and the coordinates $g_1,g_2,...$ of the associate vector ${\bf g}$ of $G_m(\ell, a_\ell)$ means that $\tau_\ell$ has $\varepsilon_1$ cycles of length $ g_1 \ell$, $\varepsilon_2$ cycles of length $ g_2 \ell$, etc.

\section{Proof of Theorems~\ref{the1} and \ref{the2}}\label{sec4}
In this section, we  present the proofs of two of our main results. Theorem~\ref{the1} provides an explicit expression for the number of $m$th roots of permutations of type ${\bf a}$, and Theorem~\ref{the2} provides a generating function for this number. We use the next notation: Let $i, j$ be any positive integers. For an $n$-permutation of Type $(i)^j$ we mean that this permutation has $j$ cycles of length $i$, and $n=ij$. First, we prove the following proposition. 
\begin{proposition}\label{npc}
Let $\ell, m \in \nn$ be fixed. Let $g \in G_m(\ell)$ and  $p \in \nn$. Let $\sigma$ be any permutation of Type $(\ell)^{gp}$. Then the number of $m$th roots of $\sigma$, $\tau$, of Type $(g\ell)^p$ is
$$
\frac{(gp)!\ell^{p(g-1)}}{g^pp!}.
$$ 
\end{proposition}
\begin{proof}
From the $g p$ cycles, we need to build $p$ cycles of length $g \ell$ for $\tau$. This we make by grouping the cycles in $p$ bundles of $g$ cycles each ($C_1, C_2, \ldots , C_{g}$), and then, we build with each bundle, a cycle of length $g\ell$. In how many ways can we  group the bundles? As no matter the order of the bundles, this problem is reduced to count the number of unordered partitions of a set of cardinality $gp$ in $p$ equal parts. This number is 
$$\frac{ (gp)!}{(g!)^{p} p!}.$$
For each of the $p$ bundles ($C_1, C_2, \ldots,  C_{g}$), we need to arrange the elements of cycles $C_1, C_2, \ldots,  C_{g}$ in a cycle $D$    that satisfies $D^m=C_1 \cdots C_{g}$. We can use the procedure in the proof of Theorem 4.8.2 \cite[\S 4.8]{wilf}), applying it to all the possible combinations of cycles $C_1, C_2, \ldots,  C_{g}$ and their elements in order to obtain all the possible different cycles $D$. With this procedure we obtain $(g  - 1)! \ell^{g - 1}$ different such cycles. As we have $p$ bundles, with $g$ disjoint cycles each, we have $((g  - 1)! \ell^{g - 1} )^{p}$ different bundles of cycles ($D_1,D_2, \cdots ,D_p$) with which we can build the different $m$th roots. Therefore the number of $m$th roots of the required Type is
$$
\frac{(gp)!}{(g!)^pp!}((g  - 1)! \ell^{g - 1})^p=\frac{(gp)!\ell^{p(g-1)}}{g^pp!}.
$$
\end{proof}
\subsection{Proof of Theorem~\ref{the1}} 
Let $m$ be a positive integer, let $\sigma$ be an $n$-permutation of type ${\bf a}$, i.e. $\sigma$ has $a_\ell$ cycles of length $\ell$ for every $\ell \in [n]$. As we see in Section~\ref{sec3}, we can build $\tau$ from $\sigma$ working separately with each part $\sigma_\ell$, $a_\ell \neq 0$. For every $\ell$-cycle in $\sigma$ we build the sets $G_m(\ell, a_\ell)$, and $\EE_m(\ell,a_\ell)$. By Proposition~\ref{props}, if $\sigma_\ell$ has an $m$th root, then $\EE_m(\ell,a_\ell)$ will be a non-empty set. We will count all the $m$th roots of $\sigma_\ell$. Take $\varepsilon$ in  $\EE_m(\ell,a_\ell)$, and for every $i \in [h]$, $h=|G_m(\ell, a_\ell)|$, we will build $\varepsilon_i$ $g_i \ell$-cycles for  $\tau_\ell$. From the $a_\ell$ cycles of $\sigma_\ell$, choose $h$ subsets of $g_i \varepsilon_i$ cycles each one, this can be made in
$$\frac{a_\ell!}{\prod_{i=1}^h  (  g_i \varepsilon_i)!  }$$
ways. Now, from every one of the $g_i \varepsilon_i$ cycles, we build $\varepsilon_i$ cycles of length $g_i \ell$. By proposition~\ref{npc} we have 
$$\frac{ (g_i\varepsilon_i)!}{(g_i!)^{\varepsilon_i} \varepsilon_i!}((g_i  - 1)! \ell^{g_i - 1} )^{\varepsilon_i}.$$
As this is for every $i\in [h]$ and all the cycles are disjoint, by the principle of multiplication we have

\begin{eqnarray*}
    a_\ell !\prod_{i=1}^h\frac{1}{(g_i \varepsilon_i)! }  \frac{(g_i \varepsilon_i)!}{ {(g_i!)}^{\varepsilon_i}\varepsilon_i!}  \left( (g_i - 1)! \ell^{g_i - 1} \right)^{\varepsilon_i},
\end{eqnarray*}
$m$th roots of $\sigma_\ell$ that consist of $\varepsilon_i$ cycles of length $g_i \ell$, for every $i \in [h]$. This expression reduces to
 \begin{eqnarray*}
     a_\ell ! \prod_{i=1}^h\frac{\ell^{(g_i - 1)\varepsilon_i}}{g_i^{\varepsilon_i}\varepsilon_i!}.
\end{eqnarray*}
Finally, we sum over all the elements in $\EE_m(\ell,a_\ell)$. Repeat the process for every cycle length $\ell$ in $\sigma$ with $a_\ell \neq 0$. So we have that the number of $m$th roots of a permutation of type ${\bf a}$ is
 \begin{eqnarray*}
 r^{(m)}({\bf a})=\prod_{\substack{\ell \geq 1\\a_\ell \neq 0}} a_\ell ! \left( \sum_{ \varepsilon \ \in \EE_m(\ell,a_\ell)} \    \prod_{i=1}^h \frac{\ell^{(g_i - 1)\varepsilon_i}}{ {g_i}^{\varepsilon_i}\varepsilon_i!}   \right).
\end{eqnarray*}
Note that $r^{(m)}({\bf a})$ is zero if and only if $ \EE_m(\ell,a_\ell)$ is an empty set for some $ \ell$, i. e. if $a_\ell$ is not a multiple of $((\ell,m))$.
\subsection{A generating function: Proof of Theorem~\ref{the2}}\label{sec5}
First we obtain a generating function for permutations of Type $(\ell)^{gp}$.

\begin{proposition}\label{npd}
Let $\ell$, $m$ be fixed positive integers. Let $g \in G_m(\ell)$. Let $p \in \nn_0$. Let $\sigma$ be any permutation of Type $(\ell)^{gp}$. Let $f(gp)$ be the number of $m$th roots of Type $(g\ell)^p$ of $\sigma$. Then
\begin{eqnarray*}
\sum_{p \geq 0} f(gp)\frac{t_\ell^{gp}}{(gp)!} =\exp\Big(\frac{\ell^{(g-1)}}{g}t_\ell^g\Big).
\end{eqnarray*}
\end{proposition}
\begin{proof}
Use Proposition~\ref{npc} to obtain

\begin{eqnarray*}
\sum_{p \geq 0} f(gp)\frac{t_\ell^{gp}}{(gp)!} &=&\sum_{p\geq 0}\frac{(gp)!\ell^{p(g-1)}}{g^pp!}\frac{t_\ell^{gp}}{(gp)!}\\
&=&\sum_{p\geq0}\frac{\ell^{p(g-1)}}{g^p}\frac{t_\ell^{gp}}{(p)!}\\
&=&\exp\Big(\frac{\ell^{(g-1)}}{g}t_\ell^g\Big).
\end{eqnarray*}
\end{proof}
Note that if $p=0$, $f(gp)=1$, but $p=0$  does not have sense in this context, we use it only for technical purposes. Now, we give a generating function for the number of $m$th roots of a permutation of type {\bf a},  for $m$ fixed. 

{\bf Theorem 2.} For $n=a_1+2a_2+\cdots +na_n$, the coefficient of $\frac{t_1^{a_1} \cdots t_{n}^{a_n}}{a_1! \cdots a_n!}$ in the expansion of
\begin{equation}\label{gfg}
\exp\Big(\sum_{\ell \geq 1}\sum_{g \in G_m(\ell)}\frac{\ell^{g-1}}{g}t_\ell^g\Big),
\end{equation}
is the number the $m$th roots of an $n$-permutation of type ${\bf a}=(a_1,...,a_n)$.

\begin{proof}
We expand equation (\ref{gfg}) and we look the coefficients of our interest. First note that for any $\ell$ the only factors that contribute to exponent of $t^{a_\ell}_\ell$ ($a_\ell \neq 0$) are the factors in
$$\prod_{g \in G_m(\ell)} e^{\frac{\ell^{g-1}}{g}t_\ell^g}=\prod_{g \in G_m(\ell)}\Big(\sum_{j\geq 0}\frac{\ell^{(g-1)j}}{g^j}t_\ell^{gj}\Big)\frac{1}{j!},$$
with $a_\ell=g_1j_1+\cdots +g_kj_k$, where $k= |G_m(\ell) |$, ${\bf g}=(g_i,...,g_k)$ is the associate vector of $G_m(\ell)$  and $j_1,...,j_k$ is any non negative solution of equation $g_1x_1+\cdots +g_kx_k=a_\ell$. The coefficient of any $t^{a_\ell}_\ell$ is
$$
\sum_{g_1j_1+\cdots +g_kj_k=a_\ell} \prod_{i=1}^k   \frac{\ell^{(g_i-1)j_i}}{g_i^{j_i}}\frac{1}{j_i!}.
$$
For the coefficient of $t_\ell^{a_{\ell}}/ a_{\ell} !$ multiply the previous expression by $a_\ell !$. Clearly the coefficient of $\frac{t_1^{a_1} \cdots t_{n}^{a_n}}{a_1! \cdots a_n!}$ comes from equation (\ref{gfg}) when we take $\ell$ from 1 to $n$ and this coefficient is
$$
\prod_{\substack{\ell \geq 1\\a_\ell \neq 0}} a_\ell ! \sum_{g_1j_1+\cdots +g_kj_k=a_\ell} \prod_{i=1}^k   \frac{\ell^{(g_i-1)j_i}}{g_i^{j_i}}\frac{1}{j_i!}.
$$
Note that this expression is equivalent to the right hand side of equation in Theorem~\ref{the1} by changing $g_1j_1+\cdots +g_kj_k=a_\ell$ in the index of summation by $\varepsilon \in \EE_m(\ell,a_\ell)$, where $\varepsilon=(j_1, ..., j_k) \in \EE_m(\ell,a_\ell)=\{\varepsilon \in \nn_0^k \;:\;g_1j_1+\cdots +g_kj_k=a_\ell\}$.
\end{proof}
Another proof can be obtained using Proposition~\ref{npd}. For fixed length $\ell$, we build the set $G_m(\ell)$, and for fixed $g \in G_m(\ell)$ we can obtain the generating function as in Proposition~\ref{npd}. Then we use properties of generating functions to obtain $\prod_{g \in G_m(\ell)}e^{\frac{\ell^{g-1}}{g}t_\ell^g}$. As for different $\ell$, the variables $t_\ell$ are different, then we can obtain the desired generating function as the product  $\prod_{\ell \geq 1}\prod_{g \in G_m(\ell)}e^{\frac{\ell^{g-1}}{g}t_\ell^g}.$

The next corollary is for the special case when $m$ is equal to a prime number $p$. Note that by Propositions~\ref{typeg2} and~\ref{typeg1}, we have that if $\gcd(\ell,p)=1$ then $G_p(\ell)=\{1,p\}$ and for $\gcd(\ell,p)=p$, $G_p(\ell)=\{p\}$. 
\begin{corollary}
Let $p$ be a fixed prime number. For $n=a_1+2a_2+\cdots +na_n$, the coefficient of $\frac{t_1^{a_1} \cdots t_{n}^{a_n}}{a_1! \cdots a_n!}$ in the expansion of

$$
\exp\Big(\sum_{i \geq 1}\frac{i^{p-1}}{p}t_i^p+\sum_{\substack{j \geq 1\\\gcd(j,p)=1}}t_j\Big)
$$
is the number of $p$-th roots of an $n$-permutation of type ${\bf a}=(a_1,...,a_n)$.
\end{corollary}

\section{A property of the probability $p_m(n)$}\label{sec6}
Let $m$ be a power of a prime number  $p$. In this section we include a proof of an interesting property about the probability, $p_m(n)$, that an $n$-permutation chosen at random has an $m$th root. In the work of B\'{o}na et al. \cite{bona}, there are a combinatorial proof of the identity $p_m(n)=p_m(n+1)$, with $m$ a prime number, and for all $n$ not congruent to $-1$ mod $m$. We prove  the analogous result for the case when $m$ is a power of prime number $q$, first by using the generating function of the number of permutations that admit an $m$th root (see \cite{cher} for a different proof). We use the terminology in \cite{wilf}. For $q \geq 1$, let $\exp_q(x)$ denote the formal series defined by
  $$
  \exp_q(x)=\sum_{i \geq 0}\frac{x^{iq}}{(iq)!}.
  $$
The following statement can be found as Theorem 4.8.3 in H. Wilf's generatingfunctionology (p. 149, \cite{wilf})
\begin{theorem}\label{thgfnr}
Let $r(n,m)$ be the number of $n$-permutations that have an $m$th root. Then
$$
\sum_{n=0}^\infty r(n,m)\frac{x^n}{n!}=\prod_{\ell=1}^\infty \exp_{((\ell,m))}\Big(\frac{x^\ell}{\ell}\Big).
$$

\end{theorem}

\begin{remark}\label{re2}
Note that if $m$ is a prime power, $m=p^r$, then for any $l \in \nn$
\begin{eqnarray*}
\left(\left( l, m \right) \right) = \begin{cases}
               1 & \text{if $\gcd(l,p) = 1$} \\
              m & \text{otherwise }
            \end{cases}
\end{eqnarray*}
\end{remark}
The next proposition gives one property of $p_m(n)$.
\begin{proposition}\label{prob}
Let $p$ be a prime.  Let $p_m(n)$ be the probability that an $n$-permutation chosen uniformly at random has an $m$th root. If $m=p^r$, $r \in \nn$, then for all $j\in \nn_0$ we have
\begin{eqnarray*}
p_m(jp)= p_m(jp+1) = \cdots = p_m(jp+(p-1))
\end{eqnarray*}
\end{proposition}
\begin{proof}
By Remark~\ref{re2} the generating function in Theorem~\ref{thgfnr} can be written as
\begin{eqnarray*}
\prod_{l \ \in \ \nn} \exp_{\left( \left( l,m \right) \right)} \left( \frac{x^l}{l}\right) =
\prod_{\substack{l \ \in \ \nn \\ \gcd(l,m) = 1}} \exp\left(\frac{x^l}{l}\right) \prod_{j=1}^{\infty}
 \exp_m  \left( \frac{x^{jp}}{jp}\right)
\end{eqnarray*}

which can be written in the form

\begin{eqnarray*}
\prod_{j = 1}^{\infty} \exp\left(\frac{x^j}{j}\right)
\prod_{j = 1}^{\infty} \exp\left(-\frac{x^{jp}}{jp}\right)
\prod_{j=1}^{\infty} \exp_m \left( \frac{x^{jp}}{jp}\right),
\end{eqnarray*}

or

\begin{eqnarray*}
\frac{1}{\left( 1-x\right)}\left( 1-x^p\right)^{1/p}
\prod_{j=1}^{\infty} \exp_m \left( \frac{x^{jp}}{jp}\right).
\end{eqnarray*}

Notice that, due to the fact that $m = p^r$, we can write $\left( 1-x^p\right)^{1/p}
\prod_{j=1}^{\infty} \exp_m \left(\frac{x^{jp}}{jp}\right)   $
as a function of the form $G\left( x \right) = \sum_{j=0}^{\infty} b_j x^{pj}$, 
for any adequate election of $b_j$. Now $\frac{1}{\left( 1-x\right)}G\left( x \right)$, where $G(x):=\left( 1-x^p\right)^{1/p}
\prod_{j=1}^{\infty} \exp_m \left( x^{jp} /jp\right) $
is the power series of some function in the variable $x^p$.
Note that $p_m(n)$ is the coefficient of $x^n$ in the expansion of
\begin{eqnarray*}
\sum_{n=0}^\infty r(n,m) \frac{x^n}{n!}=\frac{1}{\left( 1-x\right)}
G\left( x \right) = \frac{1}{\left( 1-x\right)} \sum_{k=0}^{\infty} b_k x^{kp}
\end{eqnarray*}

Now, if we take $n = kp$, for any $k=0, 1,...$  we have that
\begin{eqnarray*}
p_m(kp) &= & b_0 + b_1 + b_2 + \cdots + b_k \\
p_m(kp+ 1)& = &  b_0 + b_1 + b_2 + \cdots + b_k \\
 &\vdots   & \\
p_m(kp+ (p-1)) & = &  b_0 + b_1 + b_2 + \cdots + b_k\\
\end{eqnarray*}
\end{proof}
\subsection{A second proof of Proposition~\ref{prob}}\label{secbi}
In this section, we give a second proof of Proposition~\ref{prob}. Here we deduce Proposition~\ref{prob} from Corollary 2.12 in [4]. We start by
recalling some notation and definitions used in \cite{bona}. Let $POWER_m(n)$ be the set of $n$-permutations that have at least one $m$th root. Let $DIV_{\rho, m}(n)$ denote the set of $n$-permutations whose cycles of length a multiple of $m$ have type $\rho$ and  let $div_{\rho, m}(n)=|DIV_{\rho, m}(n)|/n!$. For $\pi \in DIV_{\rho, m}(n)$, we denote with $\pi_{(m)}$ (respectively $\pi_{(\sim m)}$) the part of $\pi$ consisting of the cycles of lengths which are (respectively not are) multiples of $m$. By Theorem~\ref{root} and Remark~\ref{re2}, if $m=p^r$ is a power of a prime $p$, then the set $POWER_{p^r}(n)$ consist of all $n$-permutations whose number of cycles of length a multiple of $p$ is divisible by $p^r$. Corollary 2.12 in \cite{bona} says that if $p$ does not divide $n+1$ then $div_{\rho, p}(n)=div_{\rho, m}(n+1)$. So in order to prove Proposition~\ref{prob} we only need to show that if $n+1$ is not a multiple of $p$ then the only possible cycle types for $\pi_{(p)}$ are the same for $POWER_m(n)$ as for $POWER_m(n+1)$. We will prove that this is the case. Let $\rho$ be any cycle type of $\pi_{(p)}$ with $\pi \in POWER_m(n)$ then, clearly, $\rho$ is a cycle type of $\pi'_{(p)}$ for some $\pi'$ in $POWER_m(n+1)$. Let $\pi \in POWER_m(n+1)$. Since $p$ does not divide $n+1$, then $\pi_{(\sim p)}$ cannot be empty. Then the $n$-permutation $\pi'$ with $\pi'_{(p)}=\pi_{(p)}$ and with the part $\pi'_{(\sim p)}$ (possible empty) containing as fixed points all the elements (different that $n+1$) in $\pi_{(\sim p)}$, has $p^r$th root and hence $\rho$ is also the cycle type of a permutation $\pi'$ in $POWER_m(n)$.\\

{\bf Acknowledgements.} The authors would like to thank E. Ugalde and L. Glebsky for many useful comments. Also the authors would like to thank the anonymous referees for its carefully reading and valuable comments and suggestions. Section~\ref{secbi} was not present in the original version. Its addition was inspired on a referee's thoughtful question. Part of this work was made when R. Moreno was staying at Campus Jalpa, UAZ. The final version was written while Rivera-Mart\'{i}nez was a postdoctoral fellow at the University of Vienna, supported by the European Research Council (ERC) grant of Goulnara Arzhantseva, grant agreement n$^\text{o}$ 259527.

{\bf Jes\'{u}s Lea\~{n}os}\\[-0.8ex]
\small Unidad Acad\'{e}mica de Matem\'{a}ticas,\\[-0.8ex]
\small Universidad Aut\'{o}noma de Zacatecas,\\[-0.8ex]
\small Zacatecas, Zac. CP 98000, M\'{e}xico.\\
\small E-mail: \texttt{jelema@uaz.edu.mx}\\
{\bf Rutilo Moreno}\\[-0.8ex]
\small Instituto de F\'{i}sica,\\[-0.8ex]
\small Universidad Aut\'{o}noma de San Luis Potos\'{i},\\[-0.8ex]
\small San Luis Potos\'{i}, SLP.  CP 78210, M\'{e}xico.\\
\small E-mail: \texttt{rutilo.moreno@gmail.com}\\
{\bf Luis Manuel Rivera-Mart\'{i}nez},\\[-0.8ex]
\small Ingenier\'{i}a El\'{e}ctrica Jalpa,\\[-0.8ex]
\small Universidad Aut\'{o}noma de Zacatecas,\\[-0.8ex]
\small E-mail: \texttt{luismanuel.rivera@gmail.com}


\begin{thebibliography}{10}
\bibitem{annin}
S. Annin, T. Jansen and C. Smith,
\newblock On $k$th roots in the symmetric and alternating groups,
\newblock {\em Pi Mu Epsilon Journal}, {\bf 12}, No. 10 (2009), 581-589.
\bibitem{blum}
J. Blum,
\newblock Enumeration of the square permutations in $\mathfrak{S}_n$,
\newblock {\em J. Comb. Theory, \/} (A) {\bf 17}, (1974), 156--161.

\bibitem{bollobas}
B. Bollob\'{a}s and B. Pittel,
\newblock The distribution of the root degree of a random permutation,
\newblock {\em Combinatorica} {\bf 29} (2) (2009), 131--151.

\bibitem{bona}
M. B\'{o}na, A. McLennan and D. White,
\newblock Permutations with roots,
\newblock {\em Random Structures $\mathfrak{E}$ algorithms,} {\bf 17} (2000), No. 2, 157--167.

\bibitem{cher}
W. W. Chernoff, 
\newblock Permutations with $p^l$-th roots, 
\newblock {\em Disc. Math.}, 125 (1994) 123--127.
\bibitem{dezahuang}
M. Deza and T. Huang,
\newblock Metrics on Permutations, a Survey,
\newblock {\em J. Comb. Inf. Sys. Sci.},  {\bf 23} (1998), 173--185.

\bibitem{glebskyrivera}
  L. Glebsky and L.\;M. Rivera,
    \newblock Almost solutions of equations in permutations,
    \newblock {\em Taiwanese J. Math.}, {\bf 13} (2009), No. 2A,  493--500. 
\bibitem{ghs}
A. Groch, D. Hofheinz, and R. Steinwandt,
    \newblock A practical attack on the root problem in braid groups,
    \newblock {\em Contemporary Math.}, {\bf 418} (2006), 121--132.

\bibitem{attila}
A. Mar\'{o}ti,
\newblock Symmetric functions, generalized blocks, and permutations with
restricted cycle structure,
\newblock {\em European J. Combin.} {\bf 28}, (2007), No. 3, 942--963.

\bibitem{pavlov1}
A. I. Pavlov,
\newblock On the number of solutions of the equation $x^k=a$ in the symmetric group $S_n$,
\newblock {\em Mat. Sb.}, {\bf 112(154)}(1980), 380--395; English {\em transl. Math. USSR Sb.}, {\bf 40} (1981).

\bibitem{pourn}
M. R. Pournaki,
 \newblock On the number of even permutations with roots,
 \newblock {\em Australas. J.Combin.}, {\bf 45} (2009), 37--12.



\bibitem{pouyanne}
N. Pouyanne,
 \newblock On the number of permutations admitting an $m$th root,
 \newblock {\em Electron. J. Comb.}, {\bf 9} (2002), $\#R3.$, 1--12.


\bibitem{schoup}
V. Schoup,
\newblock A Computational Introduction to Number Theory
and Algebra,
\newblock {\em Cambridge University Press}, 2nd edition, (2008)

\bibitem{turan}
P. Tur\'{a}n,
\newblock On some connections between combinatorics and group theory,
\newblock {\em Colloq. Math. Soc. J\'{a}nos Bolyai, P. Erd\"{o}s, A. R\'{e}nyi and V. T. S\'{o}s, eds.}, North Holland, Amsterdan (1970), 1055--1082.


\bibitem{wilf}
H. S. Wilf,
 \newblock Generatingfunctionology,
 \newblock {\em Academic Press}, San Diego, 2nd edition, (1994).

\end{thebibliography}
\end{document}